\documentclass[leqno,11pt]{amsart}
\usepackage{amssymb,amsthm,amsmath,txfonts,a4wide}
\usepackage{xcolor}
\usepackage{enumerate}
\usepackage{hyperref}

\newtheorem{defn}{Definition}[section]

\newtheorem{lem}[defn]{Lemma}
\newtheorem{thm}[defn]{Theorem}
\newtheorem{prop}[defn]{Proposition}

\newtheorem{rem}[defn]{Remark}

\numberwithin{equation}{section}

\newcommand{\eps}{\varepsilon}

\let\ptd=\partial

\let\lan=\langle
\let\ran=\rangle

\def\bN{\mathbb{N}}
\def\bP{\mathbb{P}}
\def\bR{\mathbb{R}}
\def\bT{\mathbb{T}}
\def\bZ{\mathbb{Z}}

\def\cA{\mathcal{A}}

\def\cH{\mathcal{H}}

\begin{document}
\title{Probabilistic well-posedness for supercritical wave equation on $\bT^3$}
\author{Chenmin Sun}
\address{Laboratoire de Math\'ematiques, University Paris-Sud 11, F-91405.}
\email{chenmin.sun@u-psud.fr}
\author{Bo Xia}
\address{Laboratoire de Math\'ematiques, University Paris-Sud 11, F-91405.}
\email{bo.xia@math.u-psud.fr}
\thanks{The first author is supported by Master program by University of Paris-Saclay, while the second author is supported by CSC.}
\maketitle

\begin{abstract}
  In this article, we follow the strategies, listed in \cite{Burq2011} and \cite{OhPo}, in dealing with supercritical cubic and quintic wave equations, we obtain that, the equation
  \begin{equation*}
      \left\{
        \begin{split}
        &(\ptd^2_t-\Delta)u+|u|^{p-1}u=0,\ \ 3<p<5\\
        &\big(u,\ptd_tu\big)|_{t=0}=(u_0,u_1)\in H^{s}\times H^{s-1}=:\cH^s,
      \end{split}
      \right.
    \end{equation*}
 is almost surely global well-posed in the sense of Burq and Tzvetkov\cite{Burq2011} for any $s\in (\frac{p-3}{p-1},1)$. The key point here is that $\frac{p-3}{p-1}$ is much smaller than the critical index $\frac{3}{2}-\frac{2}{p-1}$ for $3<p<5$.
\end{abstract}
%%%%%%%%%%%%%%%%%%%%%%%%%%%%%%%%%%%%%%%%%%%%%%%%%%%%%%%%%%%%%%%%%%%%%%%%%%%%%
%% introduction
%%%%%%%%%%%%%%%%%%%%%%%%%%%%%%%%%%%%%%%%%%%%%%%%%%%%%%%%%%%%%%%%%%%%%%%%%%%%%

\section{Introduction}
 In this article, we are going to construct solutions for the equation
    \begin{equation}
      \label{eq:wave}
      \left\{
        \begin{split}
        &(\ptd^2_t-\Delta)u+|u|^{p-1}u=0,\ \ 3<p<5\\
        &\big(u,\ptd_tu\big)|_{t=0}=(u_0,u_1)\in H^{s}\times H^{s-1}=:\cH^s,
      \end{split}
      \right.
    \end{equation}
    where $u$ is a real-valued function defined on $\bT^3\times\bR_t$. Via a scaling argument, one can see that $s_{cr}=\frac{3}{2}-\frac{2}{p-1}$ is a critical index in solving Equation \eqref{eq:wave}. It turns out that for $s<s_{cr}$, Equation (\ref{eq:wave}) is ill-posed, while for $s\geq s_{cr}$, Equation (\ref{eq:wave}) is well-posed (in the sense of Hardamard) only for certain range of $s$. More precisely, we have
    \begin{thm}\label{thm.classical}
    The Cauchy problem \eqref{eq:wave} is locally well-posed for data in $\cH^s$ for $s>s_{cr}$. In the opposite direction, for $p\in[3,5)$, if $s\in (0,\frac{3}{2}-\frac{2}{p-1})$, then the equation \eqref{eq:wave} is not locally well-posed in $\cH^s$. One example contradicting the continuous dependence on the initial data is as follows: there exists a sequence $(u_n)$ of global smooth solutions of \eqref{eq:wave} such that
        \[
            \lim_{n\rightarrow \infty}\|(u^{(n)}_0,u^{(n)}_1)\|_{\cH^s}=0
        \]
    but
        \[
            \lim_{n\rightarrow \infty}\|(u_n(t),\ptd_tu_n(t))\|_{L^{\infty}([0,T];\cH^s)}=\infty,\ \ \forall T>0.
        \]
\end{thm}
The well-posedness part of Theorem \ref{thm.classical} can be proved as in the work by Lindblad-Sogge\cite{Lindblad1995}, by invoking the Strichartz estimate on compact manifold due to Kapitanski\cite{kap90}. For the special case $p=3$, the equation \eqref{eq:wave} is even globally well-posed if the regularity index $s$ is sufficiently close to $1$, for the Euclidean case one can refer to works by Roy\cite{T.Roy2008}. For the ill-posedness statement of Theorem \ref{thm.classical}, one can see Burq-Tzvetkov\cite{Burq2007a}.\\

In order to overcome such ill-posedness, probabilistic tools have been introduced, by which we can construct locally and even globally well-posed solutions to several supercritical equations. This approach was first used by Bourgain\cite{Bourgain1994,Bourgain1996} to prove the invariance of Gibbs measure, introduced by Lebowitz-Rose-Speer in \cite{lebowitz1988statistical}, under the flow of the periodic nonlinear Schr\"odinger equation or 2D-defocusing nonlinear Schr\"odinger equation. By this invariance, Bourgain obtained that these equations are almost surely globally well-posed on the support of this measure. On the other hand, by randomizing the initial data via its Fourier series and a consideration of invariant measure in \cite{Burq2007a}\cite{Burq2007}, Burq-Tzvetkov proved that the cubic wave equations on the 3D unit ball are locally and globally well-posed; they also proved the local and global well-posedness of the cubic wave equation on 3D torus by a conservation law argument in \cite{Burq2011}. Using the similar argument, Burq-Thomann-Tzvetkov obtained the global existence of the cubic wave equation in higher dimension in \cite{burq2012global}. Recently Oh-Pocovnicu, by using the Wiener randomization, proved the quintic wave equation on $\bR^3$ is almost surely global well-posed with the initial data in the homogeneous space $\dot{\cH}^s(\bR^3):=\dot{H}^{s}(\bR^3)\times\dot{H}^{s-1}(\bR^3)$ with $s>\frac{1}{2}$. \\

In this article, we are going to construct solutions to Equation (\ref{eq:wave}), with $3<p<5$. And we obtained that as long as $s>\frac{p-3}{p-1}$, Equation \eqref{eq:wave} is almost surely globally well-posed. 

\begin{thm}[Almost sure global well-posedness]
  \label{sec:introduction:thm:main}
  Let $s\in(\frac{p-3}{p-1},1)$. Given $(u_0,u_1)\in\cH^s(\bT^3)$, let $(u_0^{\omega},u_1^{\omega})$ be the randomization as in (\ref{data.rando}) under the assumption (\ref{main.assum}). Then the super-critical wave equation (\ref{eq:wave}) is almost surely globally well-posed with $(u^{\omega}_0,u^{\omega}_1)$ as the initial data. More precisely, there exists a set $\Omega_{(u_0,u_1)}\subset\Omega$ of probability $1$ such that, for every $\omega\in\Omega_{(u_0,u_1)}$, there exists a unique solution $u$ (in a bounded ball around zero) to (\ref{eq:wave}) in the class:
  \begin{equation*}
    \big(S(t)(u^{\omega}_0,u^{\omega}_1),\ptd_tS(t)(u^{\omega}_0,u^{\omega}_1)\big)+C(\bR;\cH^1(\bT^3))\subset C(\bR;\cH^s(\bT^3)).
  \end{equation*}
\end{thm}

\begin{rem}
  We should notice that the lower bound $\frac{p-3}{p-1}$ is compatible with the endpoint cases $p=3$ and $p=5$. That is to say, when $p$ tends to $3$, the minimal regularity required to solve Equation (\ref{eq:wave}) becomes the one obtained in \cite{Burq2011} for the case $p=3$; and the same for the other endpoint $p=5$, see \cite{OhPo}. But if $p=3$ and $s=0$, we refer to the \cite{Burq2011} for the possible growth of Sobolev norm.
\end{rem}
  
\begin{rem}
    For the corresponding equation on Euclidean space $\bR^3$, by a similar randomization of the initial data via a unit-scale decomposition in frequency space, L\"uhrmann-Mendelson\cite{Luhrmann2013} proved the solution is globally well-posed for $1>s>\frac{p^3+5p^2-11p-3}{9p^2-6p-3}$, which is an improvement to the classical deterministic theory only when $\frac{1}{4}(7+\sqrt{73})<p<5$. And recently, they improved this result to $1>s>\frac{p-1}{p+1}$ by using Oh-Pocovnicu's ideas in \cite{OhPo}. 
\end{rem}

\begin{rem}
	For higher dimension case $d\geq4$, the global infinite energy solution to the cubic wave equation was constructed by Burq-Thomann-Tzvetkov\cite{burq2012global}, where the conditionally continuous dependence on the initial data is left unknown. But Oh-Pocovnicu succeeded to prove this uniqueness result in \cite{ohpo15}.
\end{rem}

\section*{Acknowledgement}
We should thank  Nicolas Burq for his carefully advising
 when we were preparing this paper, and also we should thank him and N. Tzvetkov for sharing the manuscripts on their recent works on quartic wave equation in the supercritical case.
\section{Preliminaries}

%%%%%%%%%%%%%%%%%%%%%%%%%%%%%%%%%%%%%%%%
%% Deterministic preliminaries
%%%%%%%%%%%%%%%%%%%%%%%%%%%%%%%%%%%%%%%%%

\subsection{Deterministic Preliminaries}
%%%%%%%%%%%%%%%%%%%%%%%%%%%%%%%%%%%%%%%%%%%%%%%%%%%%%%%%%%
%% energy estimates
%%%%%%%%%%%%%%%%%%%%%%%%%%%%%%%%%%%%%%%%%%%%%%%%%%%%%%%%%%
In this section, we recall several classical results about the linear equation
  \begin{eqnarray}\label{wave.lin}
            \left\{
                \begin{split}
                    & (\ptd^2_t-\Delta)u = f \text{ on } I\times \bT^3, \\
                    & (u,\ptd_tu)|_{t=t_0}=(u_0,u_1).
                \end{split}
            \right.
  \end{eqnarray}
  We say that $u$ solves Equation \eqref{wave.lin} on the time interval $I\ni t_0$ if $u$ satisfies for $t\in I$ the Duhamel formula
  \begin{equation*}
    u(t)=S(t-t_0)(u_0,u_1)+\int_{t_0}^{t}\frac{\sin((t-t')\sqrt{-\Delta})}{\sqrt{-\Delta}}f(t')dt',
  \end{equation*}
  where $S(t)$ is the free wave propagator defined by
  \begin{equation*}
    S(t)(u_0,u_1)=\cos(t\sqrt{-\Delta})u_0+\frac{\sin(t\sqrt{-\Delta})}{\sqrt{-\Delta}}u_1.
  \end{equation*}
  We now recall the following energy estimates for the solution $u$ to Equation \eqref{wave.lin}.
\begin{prop}[Energy estimates]
     Suppose $u$ solves Equation \eqref{wave.lin} on $I=[0,T]$. Then for any $t\in [0,T]$ we have
         \[
             \|(u(t,\cdot),\partial_tu(t,\cdot))\|_{\mathcal{H}^s} \leq C(1+T)\Big(\|(u_0,u_1)\|_{\mathcal{H}^s}+ \int_0^t \|f(r,\cdot)\|_{H^{s-1}}dr\Big).
        \]
      \end{prop}

      %%%%%%%%%%%%%%%%%%%%%%%%%%%%%%%%%%%%%%%%%
      %% Strichartz type estimate
      %%%%%%%%%%%%%%%%%%%%%%%%%%%%%%%%%%%%%%%%%%
And also, we use frequently the Strichartz estimate, which indicates the smoothing property of wave operator. In order to state this estimate, we first define the concept of "wave-admissibility" in $3$D case.
\begin{defn}
    We call a pair $(q,r)$ wave-admissible if $2\leq q\leq \infty, 2\leq r\leq \infty,(q,r)\neq(2,\infty),(q,r)\neq(\infty,2)$ and
    \[
        \frac{1}{q}+\frac{1}{r}\leq \frac{1}{2}
    \]
\end{defn}
\begin{prop}[Strichartz estimates for wave equation]\cite{Keel1998}\cite{kap90}
    Let $u$ be the solution to \eqref{wave.lin} on any time interval $0\in I\subset[0,1]$, we have
    \[
            \|u\|_{L^p(I;L^q(\bT^3))} \leq C\Big(\|(u_0,u_1)\|_{\cH^s} +\|f\|_{L^{a'}(I;L^{b'}(\bT^3))} \Big)
        \]
        under the assumptions that
        \begin{enumerate}
            \item wave admissible condition: both the pairs $(p,q)$ and $(a,b)$ are wave-admissible;
            \item Scaling invariant condition:
                \[
                    \frac{1}{p}+\frac{3}{q}=\frac{1}{a'}+\frac{3}{b'}-2=\frac{3}{2}-s.
                \]
        \end{enumerate}
\end{prop}
Indeed, in our case, the Strichartz type estimate we use is mainly for the pair $(\frac{2p}{p-3},2p)$ with regularity $s=1$ and the pair $(\infty,2)$ with $s=0$. Precisely, what we need is the following estimate
\begin{equation}
  \label{eq:strichartz:main}
     \|(u,\ptd_tu)\|_{L^{\infty}_t(I;\cH^1_x)}+\|u\|_{L^{\frac{2p}{p-3}}_t(I,L^{2p}_x)}\leq \|(u_0,u_1)\|_{\cH^1}+\|f\|_{L^1_t(I;L^2_x)}
   \end{equation}
   for any time interval $I$ containing $t_0$ with $|I|\leq 1$.In the following, we denote $\phi_0$ a radial smooth function on $\bR^3$ such that $\phi_0=1$ on the ball $B(0,1)$ and $\phi_0=0$ out side the ball $B(0,2)$. Then we recall the following projection operators for any integer $N\geq 1$
   \begin{equation*}
     P_{\leq N}u=a_0+\sum_{n\in\bZ^3\backslash\{0\}}\phi_0(\frac{n}{N})\Big(a_n\cos(n\cdot x)+b_n\sin(n\cdot x)\Big)
   \end{equation*}
   provided that $u$ is given by
   \begin{equation*}
     u=a_0 +\sum_{n\in\bZ^3\backslash{\{0\}}}a_n\cos(n\cdot x)+b_n\sin(n\cdot x).
   \end{equation*}
   When $N=2^{j}$ is a dyadic for some $j\geq 0$, we also define the projection operators
   \begin{equation*}
     P_{j}u:=P_{\leq 2^j}u-P_{\leq 2^{j-1}}u,
   \end{equation*}
   where we have used the convention that $P_{\leq 2^{-1}}u=0$. Then by the classical Littlewood-Paley theory, we have the following characterization of $H^s$-Sobolev spaces
   \begin{equation*}
     \|u\|^2_{H^s}\sim\sum_{j\geq0}2^{2js}\|P_ju\|^2_{L^2}.
   \end{equation*}
   We also have the Bernstein's inequality
   \begin{equation*}
     \|P_{\leq N}u\|_{L^q}\leq N^{\frac{3}{p}-\frac{3}{q}}\|P_{\leq N}u\|_{L^p},\ \ 1\leq p\leq q\leq\infty.
   \end{equation*}
   
%%%%%%%%%%%%%%%%%%%%%%%%%%%%%%%%%%%%
%% Probabilistic preliminaries
%%%%%%%%%%%%%%%%%%%%%%%%%%%%%%%%%%%%
\subsection{Probabilistic preliminaries}
Now let $(\alpha_j(\omega),\beta_{n,j}(\omega),\gamma_{n,j}(\omega))_{n\in \mathbb{Z}^3,\ j=0,1}$ be a series of independent identically distributed real random variables on the probability space $(\Omega,\cA,\bP)$ with the same distribution functions $\theta$. Assume that there exists $c>0$ such that
    \begin{equation}\label{main.assum}
        \forall \gamma\in \mathbb{R},\ \int_{-\infty}^{+\infty}e^{\gamma x}d\theta(x) \leq e^{c\gamma^2}.
    \end{equation}
Using such a series of random variables, we randomize the data $(u_0,u_1)\in \cH^s$, given by their Fourier series with all coefficients real
    \begin{equation}\label{fourier.series}
        u_j(x)=a_j+\sum_{n\in \mathbb{Z}_{\star}^3}\big(b_{n,j}\cos(n\cdot x)+ c_{n,j}\sin(n\cdot x) \big),\ \ \ j=0,1,\ \ \   \mathbb{Z}_{\star}^3 = \mathbb{Z}^3\backslash \{0\}
    \end{equation}
    by setting
    \begin{eqnarray}\label{data.rando}
        u_j^{\omega}(x)=\alpha_j(\omega)a_j+\sum_{n\in \mathbb{Z}_{\star}^3}\big(\beta_{n,j}(\omega)b_{n,j}\cos(n\cdot x)+ \gamma_{n,j}(\omega)c_{n,j}\sin(n\cdot x) \big).
    \end{eqnarray}
\begin{rem}
  This definition induces a Borel probability measure on $H^s$ equipped with its natural topology. Furthermore, this probability measure on $\cH^s$ has many nice properties such as "non-regularization of the data" and "non-vanishing on any open set", which exclude the possibility of "regularizing effect" originating from such procedure when applied to PDE. See \cite{Burq2011}\cite{Burq2007a} for more details.
\end{rem}
We first recall the following probabilistic estimates for any given $\ell^2$ sequence $(c_n)$, which is very important in obtaining probabilistic estimates for the random variables $(u^{\omega}_0,u^{\omega}_1)$.
\begin{prop}\cite{Burq2007a}
  \label{prop:proba:esti}
  Let $\{g_n\}$ be a sequence of mean-zero, real-valued random variables and $g_n$ satisfies the assumption \eqref{main.assum} for any integer $n$. Then for any $\ell^2$ sequence $(c_n)$ and any $q\geq2$, there exists $c>0$ such that
  \begin{equation*}
    \|\sum g_n(\omega)c_n\|_{L^q_{\omega}}\leq c\sqrt{q}\|(c_n)\|_{\ell^2}.
  \end{equation*}
\end{prop}

By using this estimates, we can prove the following local-in-time probabilistic Strichartz estimates by using the ideas used in \cite{Burq2007a}\cite{Burq2007}\cite{Poc}.
\begin{prop}\cite{Burq2007a}\cite{Burq2007}\cite{Poc}
  \label{prop:local-in-time:strich}
  Let $(u_0,u_1)\in \cH^s(\bT^3)$ be given by the series (\ref{fourier.series}) with all coefficients real and $(u^{\omega}_0,u^{\omega}_1)$ be randomized as in (\ref{data.rando}). Assume $I=[a,b]\subset \bR$ is a compact time interval.
  \begin{enumerate}[(i)]
  \item If $s=0$, then for any given $1\leq q<\infty$ and $2\leq r<\infty$, there exists $C,c>0$ such that
    \begin{equation*}
      \bP\Big(\|S(t)(u^{\omega}_0,u^{\omega}_1)\|_{L^{q}_tL^{r}_x(I\times\bT^3)}>\lambda\Big)\leq C\exp\Big(-c\frac{\lambda^2}{|I|^{\frac{2}{q}}\|(u_0,u_1)\|_{\cH^{0}}^2}\Big).
    \end{equation*}
  \item For any given $1\leq q<\infty, 2\leq r\leq\infty$, there exist $C,c>0$ such that
    \begin{equation*}
      \bP\Big(\|S(t)(u^{\omega}_0,u^{\omega}_1)\|_{L^q_tL^r_x(I\times\bT^3)}>\lambda\Big)\leq C\exp\Big(-c\frac{\lambda^2}{|I|^{\frac{2}{q}}\|(u_0,u_1)\|^2_{\cH^s}}\Big)
    \end{equation*}
    for $(ii.a)$ $s=0$ if $r<\infty$ and $(ii.b)$ $s>0$ if $r=\infty$.
  \end{enumerate}
\end{prop}

By denoting $\tilde{S}(t)$ by
\begin{equation}
  \label{eq:defn:tilde:S}
  \tilde{S}(t)(u_0,u_1):=-\frac{|\nabla|}{\lan\nabla\ran}\sin(t|\nabla|)u_0+\frac{\cos(t|\nabla|)}{\lan\nabla\ran}u_1,
\end{equation}
we state the following proposition, which plays an important role in obtaining the probabilistic \textit{a priori} bound on the the solution to {{Equation}} (\ref{eq:truncated:v}). 
\begin{prop}
  \label{prop:long-time:strich}
  Let $(u_0,u_1)\in \cH^s(\bT^3)$ be given by the series (\ref{fourier.series}) with all coefficients real and $(u^{\omega}_0,u^{\omega}_1)$ be randomized as in (\ref{data.rando}). And let $T>0$ and $S^{\ast}(t)=S(t)$ or $\tilde{S}(t)$. Then for $2\leq r\leq \infty$, we have
  \begin{equation}
    \label{eq:long-time:strich}
    \bP\Big(\|S^{\ast}(t)(u^{\omega}_0,u^{\omega}_1)\|_{L^{\infty}_tL^{r}_x([0,T]\times\bT^3)}>\lambda\Big)\leq C \exp\Big(-c\frac{\lambda^2}{\max(1,T^2)\|(u_0,u_1)\|_{\cH^{\eps}}^2}\Big)
  \end{equation}
  for any $\eps>0$, where the constants $C$ and $c$ depend only on $r$ and $\eps$.
\end{prop}

The proof of Proposition \ref{prop:long-time:strich} runs the same as what T.Oh and O. Pocovnicu did in \cite{OhPo}. However, by viewing $\lan\ptd_t\ran^{\eps}=\lan\nabla\ran^{\eps}$ when acting on $e^{\pm it\sqrt{-\Delta}}u_0$, we can prove Proposition \ref{prop:long-time:strich} by the trick of loss of derivatives in space-time. See \cite{BuLe} for more details.

%%%%%%%%%%%%%%%%%%%%%%%%%%%%%%%%%%%%%%%%%%%%%%%%%%
%% Probabilistic analysis of NLW
%%%%%%%%%%%%%%%%%%%%%%%%%%%%%%%%%%%%%%%%%%%%%%%%%%

\section{Probabilistic Analysis of NLW}
We first look at the truncated equation
\begin{equation}
  \label{eq:truncated}
  \left\{
    \begin{split}
    &(\ptd^2_t-\Delta)u_N+|u_N|^{p-1}u_N=0\\
    &\big(u_N,\ptd_tu_N\big)=(P_{\leq N}u_0,P_{\leq N}u_1).
  \end{split}
  \right.
\end{equation}
As the initial data $(P_{\leq N}u_0,P_{\leq N}u_1)$ is smooth for any data $(u_0,u_1)\in\cH^{s}$, Equation \ref{eq:truncated} has global smooth solution. In order to study the contributions of the high-frequency portion of the initial data, we rewrite Equation \ref{eq:truncated} as
\begin{equation}
  \label{eq:truncated:v}
  \left\{\begin{split}
    &(\ptd^2_t-\Delta)v_N+|v_N+z_N|^{p-1}(v_N+z_N)=0\\
    &\big(v_N,\ptd_tv_N\big)=(0,0),
  \end{split}
  \right.
\end{equation}
where $z_N=S(t)(P_{\leq N}u_0,P_{\leq N}u_1)$ is the free wave propagation of $(P_{\leq N}u_0,P_{\leq N}u_1)$. Then we have
\begin{prop}
  \label{lemma:key:energy}
  Let $s\in(\frac{p-3}{p-1},1)$ and $N\geq1$ dyadic. Given $T,\eps>0$, there exists $\tilde{\Omega}_{N,T,\eps}\subset\Omega$ such that
  \begin{enumerate}[(i)]
  \item $\bP(\tilde{\Omega}^c_{N,T,\eps})<\eps$,
  \item There exists a finite constant $C(T,\eps,\|(u_0,u_1)\|_{\cH^s})$, independent of $N$, such that the following energy bound holds
    \begin{equation}
      \label{eq:energy:bound}
      \sup_{t\in[0,T]}\|(v^{\omega}_N(t),\ptd_tv^{\omega}_N(t)\|_{\cH^1}\leq C(T,\eps,\|(u_0,u_1)\|_{\cH^s})
    \end{equation}
    for any solutions $v^{\omega}_N$ to \eqref{eq:truncated:v} with $\omega\in \tilde{\Omega}_{N,T,\eps}$.
  \end{enumerate}
\end{prop}
\begin{rem}
  \label{rem:vN:nonN}
  Indeed, we can even choose the set $\tilde{\Omega}_{N,T,\eps}$ independent of $N$, which is just a careful application of propositions \ref{prop:local-in-time:strich} and \ref{prop:long-time:strich}.
\end{rem}
\begin{proof} We argue in the same way as Oh-Pocovnicu did in \cite{OhPo}. First observe that
  \begin{equation*}
    \|v_N^{\omega}\|_{L^2}\leq c\|v^{\omega}_N\|_{L^{p+1}}\leq c E(v^{\omega}_N)^{\frac{1}{p+1}}.
  \end{equation*}
  Now if we have
  \begin{equation}
    \label{eq:assup:energy}
    \sup_{t\in[0,T]}E(v^{\omega}_N)\leq C
  \end{equation}
  then we would have
  \begin{equation*}
    \sup_{t\in[0,T]}\|(v^{\omega}_N(t),\ptd_tv^{\omega}_N(t)\|_{\cH^1}^2\leq (C+C^{\frac{2}{p+1}}).
  \end{equation*}
  Consequently, we only need to prove \eqref{eq:assup:energy}.

  As above $z_N(t)=S(t)(P_{\leq N}u_0,P_{\leq N}u_1)$ and $\lan \nabla \ran \tilde{z}_N=\ptd_tz_N$. Let $\delta >0$ sufficiently small such that $\frac{p-3}{p-1}+\delta<s$. For fixed $T,\eps>0$ we define $\tilde{\Omega}_{N,T,\eps}$ by
  \begin{equation*}
    \tilde{\Omega}_{N,T,\eps}:=\{\omega:\|z_N^{\omega}\|_{L^{2p}_{T,x}}+\|z^{\omega}_N\|_{L^{\infty}_TL^{p+1}_x}+\|z^{\omega}_N\|_{L^{\infty}_tL^{\frac{4(p+1)}{5-p}}}^2+\|\lan\nabla\ran^{s-}\tilde{z}^{\omega}_N\|_{L^{\infty}_{T,x}}\leq \lambda\}
  \end{equation*}
  where $\lambda=\lambda(T,\eps,\|(u_0,u_1)\|_{\cH^s})>0$ is chosen such that $\bP(\tilde{\Omega}^c_{N,T,\eps})<\eps$. The existence of $\tilde{\Omega}_{N,T,\eps}$ is guaranteed by Lemma \ref{prop:local-in-time:strich} and Lemma (\ref{eq:long-time:strich}).

  In the following, we are going to prove
  \begin{equation}
    \label{eq:energy:bound:E}
    \sup_{t\in[0,T]}E(v^{\omega}_N(t))\leq C(T,\eps,\|(u_0,u_1)\|_{\cH^s})
  \end{equation}
  for $\omega\in\tilde{\Omega}_{N,T,\eps}$.
  In the following of this section, we suppress the index $N$ for the solution $v_N$ to Equation \eqref{eq:truncated:v}. Thus to achieve the energy bound \eqref{eq:energy:bound:E}, we differentiate the expression of the energy and obtain
  \begin{eqnarray*}
    \frac{d}{dt}E(v)(t)&=&\int_{\bT^3}\ptd_tv(\ptd^2_tv-\Delta v+|v|^{p-1}v)dx\\
                       &=&-\int_{\bT^3}\ptd_tv(|v+z|^{p-1}(v+z)-|v|^{p-1}v)dx\\
                       &=&-\int_{\bT^3} \ptd_tv(p|v|^{p-1}z+p(p-1)|v+\theta z|^{p-2}z^2)dx
  \end{eqnarray*}
  where in the last equality we have used differential mean value equality with $\theta\in[0,1]$.
  By integrating in time, we have
  \begin{eqnarray*}
    E(v)(t) &=& E(v)(0)-\int_0^t\int_{\bT^3}\ptd_tv(t')[pz(t')|v(t')|^{p-1}+p(p-1)|v(t')+\theta z(t')|^{p-2} z(t')^2]dt'dx\\
            &=&-\int_{\bT^3}\int^t_0z(t')\ptd_t[|v|^{p-1}v(t')]dt'dx-\int^t_0\int_{\bT^3}\ptd_tv(t')p(p-1)|v(t')+\theta z(t')|^{p-2} z(t')^2dt'dx\\
            &=:& I(t)+II(t).
  \end{eqnarray*}

  Noticing that
  \begin{eqnarray*}
    ||v+\theta z|^{p-2} z^2|\leq c(|v|^{p-2} z^2+|z|^p), 
  \end{eqnarray*}
  where $c$ is a constant depending only on $p$,  we have
  \begin{eqnarray*}
    |II(t)|&\leq& \int_0^t\|\ptd_tv(t')\|_{L^2}\|v(t')\|^{p-2}_{L^{p+1}}\|z\|^2_{L^{\frac{4(p+1)}{5-p}}}(t')+\int_0^t\|\ptd_tv(t')\|_{L^2}\|z(t')\|^{2p}_{L^{2p}}dt'\\
      &\leq& (1+\|z\|^2_{L^{\infty}_tL^{\frac{4(p+1)}{5-p}}})\int^t_0\max\big(E(v)(t'),E(v)^{\frac{3(p-1)}{2(p+1)}}\big)dt'+\|z\|^{2p}_{L^{4p}_TL^{2p}_x}.
  \end{eqnarray*}
  Thus thanks to $p<5$, we have that $\frac{3p-3}{2p+2}\leq 1$. And hence we only need to consider
  \begin{equation}
    \label{eq:bound:II}
    |II(t)|\leq (1+\|z\|^2_{L^{\infty}_tL^{\frac{4(p+1)}{5-p}}})\int^t_0E(v)(t')dt'+\|z\|^{2p}_{L^{4p}_TL^{2p}_x}.
  \end{equation}

  Now, we are going to deal with the term $I(t)$. As $v(0)=0$ and $v=v_N^{\omega}$ is smooth, both in $t$ and $x$, integrating by parts, we have
  \begin{equation}
    \label{eq:bound:I}
    I(t)=-\int_{\bT^3}z(t)|v|^{p}+\int_{\bT^3}\int^t_0\ptd_tz(t')|v(t')|^pdt'dx=:I_1(t)+I_2(t).
  \end{equation}
  As for the first term $I_1(t)$, we have
  \begin{equation}
    \label{eq:bound:I1}
    |I_1(t)|\leq a\|v(t)\|^{p+1}_{L^{p+1}}+a^{-p}\|z(t)\|^{p+1}_{L^{p+1}}\leq aE(v)(t)+a^{-p}\|z\|_{L^{\infty}_TL_x^{p+1}}^{p+1},
  \end{equation}
  where $a$ is a small constant, to be chosen later.

  To bound the term $I_2(t)$, we need the following lemma:
\begin{lem}  
  \label{lem:bound:I2}
  Let $v,\tilde{z}$ as above, we have
  \begin{equation*}
    |\int_{\bT^3}|v|^{p-1}v\lan \nabla\ran \tilde{z}dx|\leq (\|\lan \nabla\ran ^{s-}z\|_{L^{\infty}_x}+1)E(v)(t)+\|\lan\nabla\ran^{s-}\tilde{z}\|_{L^{\infty}_x}^{p+1},
  \end{equation*}
  where $s-:=s-\delta$ for any sufficiently small, positive $\delta$.  
\end{lem}

\begin{proof}[Proof of Lemma \ref{lem:bound:I2}]
  Denote $P_j$ the Littlewood-Paley projection onto the dyadic $2^j$ for $j\in \bN^+$. Then we have
  \begin{equation*}
   \int_{\bT^3}|v|^{p-1}v\lan \nabla\ran \tilde{z}dx\sim \sum_{k=-1}^{k=1}\sum_{j\geq 0}\int_{\bT^{3}}P_{j+k}(|v|^{p-1}v)P_j(\lan\nabla\ran z)dx.
 \end{equation*}
 Notice that the contribution of the summation over $k=-1,0,1$ can be bounded by that of the case $k=0$, so in the following we will omit the summatin over the index $k$ and sometimes omit the index $k$ directly.
 
 For the low frequency case $j\leq 2$, we have
 \begin{equation*}
   |\int_{\bT^{3}}P_{j+k}(|v|^{p-1}v)P_j(\lan\nabla\ran\tilde{z})dx|\leq \|\lan\nabla\ran^{s-}\tilde{z}\|_{L^{\infty}_x}\|v\|^{p}_{L^{p+1}}.
 \end{equation*}
 A further application of Young's or H\"older inequality, we have
 \begin{equation}
   \label{eq:bound:low:fr}
   \|\sum_{j\leq 2}\int_{\bT^{3}}P_{j+k}(|v|^{p-1}v)P_j(\lan\nabla\ran \tilde{z})\|\leq \|\lan\nabla\ran^{s-}\tilde{z}\|_{L^{\infty}_x}^{p+1}+E(v)(t).
 \end{equation}
 
 For the high frequency portion $j>2$, we split the nonlinear part $P_{j}(|v|^{p-1}v)$ into the small value part and large value part. Precisely, we introduce a bump function $\chi:\bR^{+}\rightarrow [0,1]$, which takes its value $1$ on $[0,1]$ and vanishes outside $[0,2]$, then we split
 \begin{equation*}
   P_j(|v|^{p-1}v)=P_j\Big(|v|^{p-1}v\chi(\frac{v^2}{\lambda_j^2})\Big)+P_j\Big(|v|^{p-1}v\big(1-\chi(\frac{v^2}{\lambda_j^2})\big)\Big)=:I_{21}+I_{22},     
 \end{equation*}
 where $\lambda_j$ is a sequence of numbers to be chosen later.

 For small values of $v$, by H\"older inequality and Bernstein type estimates, we can do the following calculations
 \begin{eqnarray*}
   |\int_{\bT^3}P_{j}(|v|^{p-1}v\chi)P_j(\lan\nabla\ran z)dx| &=& |\int_{\bT^3}P_{j}(|v|^{p-1}v\chi)\nabla\cdot \nabla^{-1}P_j(\lan\nabla\ran z)dx|\\
                                                              &=&|\int_{\bT^3}\nabla P_{j}(|v|^{p-1}v\chi) \nabla^{-1}P_j(\lan\nabla\ran z)dx|\\
                                                              &\lesssim& 2^{-j(s-)}\|\lan\nabla\ran^{s-}\tilde{z}\|_{L^{\infty}_x}\|P_{j}(\nabla(|v|^{p-1}v\chi))\|_{L^1}\\
                                                              &\lesssim& 2^{-j(s-)}\|\lan\nabla\ran^{s-}\tilde{z}\|_{L^{\infty}_x}\|v^{p-1}\nabla v \chi\|_{L^1_x}\\
                                                              &\lesssim& 2^{-j(s-)}\|\lan\nabla\ran^{s-}\tilde{z}\|_{L^{\infty}_x} \||v|^{p-1-\frac{p+1}{2}}\chi\|_{L^{\infty}_x}\|\nabla v\|_{L^2}\||v|^{\frac{p+1}{2}}\|_{L^2}\\
   &\lesssim& 2^{-j(s-)}\|\lan\nabla\ran^{s-}\tilde{z}\|_{L^{\infty}_x} \lambda_j^{\frac{p-3}{2}}E(v)(t).
 \end{eqnarray*}
 To guarantee the convergence of the series $\sum_{j\geq 2}2^{-j(s-)}\lambda_j^{\frac{p-3}{2}}$, we choose $\lambda_j=2^{aj}$ with $a\in (0,\frac{2s-}{p-3})$. And in this case, we have
 \begin{equation}
   \label{eq:bound:small:v}
    |\sum_{j>2}\int_{\bT^3}P_{j}(|v|^{p-1}v\chi)P_j(\lan\nabla\ran z)dx|\lesssim \|\lan\nabla\ran^{s-}\tilde{z}\|_{L^{\infty}_x}E(v)(t)
  \end{equation}
  provided that the Sobolev regularity index $s$ is positive.

 For the case $v$ is large, we first consider the case $[p]$ is odd. By denoting $\alpha=p-[p]$, we do the following calculations
 \begin{eqnarray*}
   |\int_{\bT^3}P_{j}\big(|v|^{p-1}v(1-\chi)\big)P_j(\lan\nabla\ran z)dx| &\leq& |\int_{\bT^3}P_j\big(\sum_{j_1,j_2,\dots,j_{[p]-1},\nu}\Pi_{i=1}^{[p]-1}P_{j_i}vP_{\nu}(v|v|^{\alpha}(1-\chi))\big)P_j(\lan\nabla\ran\tilde{z})dx|\\
                                                                          &\leq& \|P_j\lan\nabla\ran^{s-}\tilde{z}\|_{L^{\infty}_x}2^{j(1-(s-))}\sum_{j_1,\dots,j_{[p]-1},\nu}\Big\|\Pi_{i=1}^{[p]-1}P_{j_i}vP_v(v|v|^{\alpha}(x-\chi))\Big\|_{L^1_x}\\
   &\leq& M_j+N_j
 \end{eqnarray*}
 where
 \begin{equation*}
   M_j:= \|P_j\lan\nabla\ran^{s-}\tilde{z}\|_{L^{\infty}_x}2^{j(1-(s-))}\sum_{j_1\geq{\max{(j_2,\dots,j_{[p]-1},\nu)}}}\Big\|\Pi_{i=1}^{[p]-1}P_{j_i}vP_v(v|v|^{\alpha}(x-\chi))\Big\|_{L^1_x}
 \end{equation*}
 and
 \begin{equation*}
   N_j:= \|P_j\lan\nabla\ran^{s-}\tilde{z}\|_{L^{\infty}_x}2^{j(1-(s-))}\sum_{\nu\geq{\max{(j_1,\dots,j_{[p]-1})}}}\Big\|\Pi_{i=1}^{[p]-1}P_{j_i}vP_v(v|v|^{\alpha}(x-\chi))\Big\|_{L^1_x}.
 \end{equation*}
 \begin{enumerate}
 \item To control $M_j$: observe that if $j\gg j_1$, we should have that $M_j=0$. And hence, we have
   \begin{eqnarray*}
     \sum_{j>2}M_j &\leq& \sum_{j>2} \|P_j\lan\nabla\ran^{s-}\tilde{z}\|_{L^{\infty}_x}2^{j(1-(s-))}\sum_{\substack{j_1\geq{{j_2,\dots,j_{[p]-1},}} \\ j_1\geq \nu\\  j_1+[p]\geq j}}\Big\|\Pi_{i=1}^{[p]-1}P_{j_i}vP_v(v|v|^{\alpha}(x-\chi))\Big\|_{L^1_x}\\
     &\leq&  \sum_{j>2}\|P_j\lan\nabla\ran^{s-}\tilde{z}\|_{L^{\infty}_x}2^{j(1-(s-))}\sum_{\substack{j_1\geq{{j_2,\dots,j_{[p]-1}}} \\ j_1\geq \nu\\  j_1+[p]\geq j}}\|P_{j_1}v\|_{L^{\frac{p+1}{2}}_x}\|\Pi_{i=2}^{[p]-1}P_{j_i}v\|_{L^{\frac{p+1}{[p]-2}}_x}\|P_{\nu}(v|v|^{\alpha}(1-\chi)\|_{L^{\frac{p+1}{1+\alpha}}_x}\\
     &\leq& \sum_{j>2}\|P_j\lan\nabla\ran^{s-}\tilde{z}\|_{L^{\infty}_x}2^{j(1-(s-))} \sum_{j_1+[p]\geq j}\|P_{j_1}v\|_{L^{\frac{p+1}{2}}_x}\|v\|_{L^{p+1}_x}^{p-1}\\
                   &\leq& \sum_{j>2}\|P_j\lan\nabla\ran^{s-}\tilde{z}\|_{L^{\infty}_x}2^{j(1-(s-))}\sum_{j_1+[p]\geq j}\|P_{j_1}v\|^{\frac{2}{p-1}}_{L^2}\|P_{j_1}v\|^{\frac{p-3}{p-1}}_{L^{p+1}}\|v\|_{L^{p+1}_x}^p\\
                   &\lesssim& \sum_{j>2}\|P_j\lan\nabla\ran^{s-}\tilde{z}\|_{L^{\infty}_x}2^{j(1-(s-))}\sum_{j_1+[p]\geq j}2^{-j\frac{2}{p-1}}\|P_{j_1}\nabla v\|^{\frac{2}{p-2}}_{L^2}\|v\|_{L^{p+1}_x}^{p-\frac{2}{p-1}}\\
                   &\lesssim& \sum_{j>2}\|\lan\nabla\ran^{s-}\tilde{z}\|_{L^{\infty}_x}2^{j(1-(s-))}2^{-j(\frac{2}{p-1}-)}E(v).
   \end{eqnarray*}
   Consequently, the last series converges provided
   \begin{equation*}
     s>\frac{p-3}{p-1}.
   \end{equation*}
   And in this case, we have
   \begin{equation}
     \label{eq:bound:Mj}
     \sum_{j>2}M_j\lesssim \|\lan\nabla\ran^{s-}\tilde{z}\|_{L^{\infty}_x}E(v)(t).
   \end{equation}
 \item To control $N_j$: the same observation as in controlling $M_j$ allows us to only need to deal with the case $\nu+[p]\geq j$. Then
   \begin{eqnarray*}
     \sum_{j>2}N_j &=& \sum_{j>2}  \|P_j\lan\nabla\ran^{s-}\tilde{z}\|_{L^{\infty}_x}2^{j(1-(s-))}\sum_{\substack{\nu\geq j_1,\dots,j_{[p]-1}\\ \nu+[p]\geq j}}\Big\|\Pi_{i=1}^{[p]-1}P_{j_i}vP_v(v|v|^{\alpha}(x-\chi))\Big\|_{L^1_x}\\
                   &=& \sum_{j>2} \|P_j\lan\nabla\ran^{s-}\tilde{z}\|_{L^{\infty}_x}2^{j(1-(s-))}\sum_{\substack{\nu\geq j_1,\dots,j_{[p]-1}\\ \nu+[p]\geq j}}\|P_{\nu}(v|v|^{\alpha}(1-\chi))\|_{L^{\frac{p+1}{2+\alpha}}}\|\Pi_{i=1}^{[p]-1}P_{j_i}v\|_{L^{\frac{p+1}{[p]-1}}_x}\\
                   &\leq&\sum_{j>2}\|\lan\nabla\ran^{s-}\tilde{z}\|_{L^{\infty}_x}2^{j(1-(s-))}\sum_{\nu+[p]\geq j}\|P_v(v|v|^{\alpha}(1-\chi))\|_{L^{\frac{2(p+1)}{p+1+2\alpha}}_x}^{\frac{2}{p-1}}\|P_{\nu}(v|v|^{\alpha}(1-\chi))\|^{\frac{p-3}{p-1}}_{L^{\frac{p+1}{1+\alpha}}_x}\|v\|_{L^{p+1}_x}^{[p]-1}\\
                   &\leq&\sum_{j>2}\|\lan\nabla\ran^{s-}\tilde{z}\|_{L^{\infty}_x}2^{j(1-(s-))}\sum_{\nu+[p]\geq j}\|P_v(v|v|^{\alpha}(1-\chi))\|_{L^{\frac{2(p+1)}{p+1+2\alpha}}_x}^{\frac{2}{p-1}}\|v\|^{p-\frac{2(1+\alpha)}{p-1}}_{L^{p+1}_x}\\
     &\lesssim& \sum_{j>2}\|\lan\nabla\ran^{s-}\tilde{z}\|_{L^{\infty}_x}2^{j(1-(s-))}\sum_{\nu+[p]\geq j}2^{(-v\frac{2}{p-1})+}\|\nabla P_v(v|v|^{\alpha}(1-\chi))\|_{L^{\frac{2(p+1)}{p+1+2\alpha}}_x}^{\frac{2}{p-1}}\|v\|^{p-\frac{2(1+\alpha)}{p-1}}_{L^{p+1}_x}
   \end{eqnarray*}
   Since
   \begin{equation*}
     \nabla P_{\nu}(v|v|^{\alpha}(1-\chi))\sim \nabla v |v|^{\alpha},
   \end{equation*}
   we have that
   \begin{eqnarray*}
     \sum_{j>2}N_j &\lesssim& \sum_{j>2}\|\lan\nabla\ran^{s-}\tilde{z}\|_{L^{\infty}_x}2^{j(1-(s-))}\sum_{\nu+[p]\geq j}2^{(-v\frac{2}{p-1})+}\|\nabla v |v|^{\alpha}\|_{L^{\frac{2(p+1)}{p+1+2\alpha}}_x}^{\frac{2}{p-1}}\|v\|^{p-\frac{2(1+\alpha)}{p-1}}_{L^{p+1}_x}\\
                   &\lesssim& \sum_{j>2}\|\lan\nabla\ran^{s-}\tilde{z}\|_{L^{\infty}_x}2^{j(1-(s-))}\sum_{\nu+[p]\geq j}2^{(-v\frac{2}{p-1})+}\|\nabla v\|^{\frac{2}{p-1}}_{L^2} \||v|^{\alpha}\|_{L^{\frac{p+1}{\alpha}}_x}^{\frac{2}{p-1}}\|v\|^{p-\frac{2(1+\alpha)}{p-1}}_{L^{p+1}_x}\\
                   &\lesssim& \sum_{j>2}\|\lan\nabla\ran^{s-}\tilde{z}\|_{L^{\infty}_x}2^{j(1-(s-))}\sum_{\nu+[p]\geq j}2^{(-v\frac{2}{p-1})+}E(v).
   \end{eqnarray*}
   Thus the last series converges provided that
   \begin{equation*}
     s>\frac{p-3}{p-1}.
   \end{equation*}
   And in this case we have
   \begin{equation}
     \label{eq:bound:N_j}
     \sum_{j>2}N_j\lesssim \|\lan\nabla\ran^{s-}\tilde{z}\|_{L^{\infty}_x}E(v)(t).
   \end{equation}
 \end{enumerate}

 For the case $[p]$ is even, we should replace the expression $P_j(|v|^{p-1}v)=\sum P_j(\Pi_{i=1}^{[p]-1}P_{j_i}vP_{\nu}(v|v|^{\alpha}))$ in the case that $[p]$ is odd by the expression $P_{j}(|v|^{p-1}v)=\sum P_j(\Pi_{i=1}^{[p]-2}P_{j_i}vP_{\nu}(v|v|^{1+\alpha}))$, and do the same calculations as above with some different H\"older indices.

 Now, in our situation, it is only left to prove the case $\alpha=0$, which is just the case $p=4$. Indeed, this case is much easier to check.

 By collecting the bounds \eqref{eq:bound:low:fr}, \eqref{eq:bound:small:v}, \eqref{eq:bound:Mj} and \eqref{eq:bound:N_j}, we can close the proof of Lemma \ref{lem:bound:I2}.
 \end{proof}

As a consequence of Lemma \ref{lem:bound:I2}, by the fact that $\ptd_tz(t)=\lan \nabla\ran \tilde{z}$, we have
\begin{equation}
  \label{eq:bound:I2}
  |I_2|\leq \int^t_0\|\lan \nabla\ran^{s-}\tilde{z}\|_{L^{\infty}_x}(t')^{p+1}(1+\|\lan\nabla\ran^{s-}\tilde{z}\|_{L^{\infty}_x}(t'))E(v)(t')dt'.
\end{equation}

Finally, by collecting the estimates \eqref{eq:bound:II}, \eqref{eq:bound:I1} and \eqref{eq:bound:I2} together, with $a$ sufficiently small, and using Gronwall's lemma, one can finish the proof of Proposition \ref{lemma:key:energy}
\end{proof}

%%%%%%%%%%%%%%%%%%%%%%%%%%%%%%%%%%%%%%%%%%%%%%
%% Deterministic analysis of NLW
%%%%%%%%%%%%%%%%%%%%%%%%%%%%%%%%%%%%%%%%%%%%%%

\section{Deterministic analysis of NLW}   
Using energy and Strichartz estimates, we can establish the following lemma, which is the key deterministic step in constructing solutions for the equation \eqref{eq:wave}.
\begin{lem}\label{lem:key:loc}
  Given any $p\in(3,5)$, for the wave equation 
  \begin{equation}
    \label{eq:wave:ini}
    \left\{
      \begin{split}
      &(\ptd^2_t-\Delta)v+|v+f|^{p-1}(v+f)=0,\\
      &\big(v,\ptd_tv\big)|_{t=t_0}=(v_0,v_1)\in\cH^1(\bT^3),
    \end{split}
    \right.
  \end{equation}
  there exists $t_{\ast}>0$, such that the equation \eqref{eq:wave:ini} has a unique solution in $\Big(C([t_0,t_0+t_{\ast}];H^1_x)\cap L^{\frac{2p}{p-3}}_t(I;L^{2p}_x)\Big)\times C([t_0,t_0+t_{\ast}];L^2_x)=:X$, under the condition that
  \begin{equation}
    \label{cond:f}
  \|f\|_{L^{\frac{2p}{p-3}}_t([t_0,t_0+t_{\ast}];L^{2p}_x)}\leq Kt_{\ast}^{\beta}.
\end{equation}
where $\beta$ is some positive number.
\end{lem}
\begin{rem}
  Due to the fact that $p$ is strictly less than $5$, we do not need to prove Lemma \ref{lem:key:loc} via the stability theory for the critical NLW as Pocovnicu did in \cite{Poc}.
\end{rem}
\begin{proof}
  We use fixed point argument on the closed ball $B(0,R)\subset X$ for some to-be-selected radius $R$. We define the map $L$ on $B(0,R)$ in the way
  \[
    L\colon v\in B(0,R)\rightarrow u  
  \]
  where $u$ solves the equation
  \begin{equation*}
    \left\{
      \begin{split}
      & (\ptd^2_t-\Delta)u+|v+f|^{p-1}(v+f)=0,\\
      & \big(v,\ptd_tv\big)|_{t=t_0}=(v_0,v_1).
    \end{split}
    \right.
  \end{equation*}

  The estimates for $u,v\in B(0,R)$
  \begin{eqnarray*}
    \|Lv\|_X&\leq& \|(v_0,v_1)\|_{\cH^1}+t_{\ast}^{\frac{5-p}{2}}\Big(\|v\|^p_{L^{\frac{2p}{p-3}}_tL^{2p}_x}+\|f\|^{p}_{L^{\frac{2p}{p-3}}_tL^{2p}_x}\Big)\\
    &\leq&\|(v_0,v_1)\|_{\cH^1}+t_{\ast}^{\frac{5-p}{2}}\Big(\|v\|^p_{L^{\frac{2p}{p-3}}_tL^{2p}_x}+K^pt_{\ast}^{p\beta}\Big),
  \end{eqnarray*}
  together with
  \begin{eqnarray*}
    \|Lu-Lv\|_X&\leq& t_{\ast}^{\frac{5-p}{2}}\|u-v\|_X\Big(\|u\|^{p-1}_X+\|v\|^{p-1}_X+\|f\|^{p-1}_{L^{\frac{2p}{p-3}}_tL^{2p}_x}\Big)\\
    &\leq& t_{\ast}^{\frac{5-p}{2}}\|u-v\|_X\Big(\|u\|^{p-1}_X+\|v\|^{p-1}_X+K^{p-1}t_{\ast}^{(p-1)\beta}\Big)
  \end{eqnarray*}
  indicate that the map $L$ is a contraction map onto $B(0,R)$, provided that
  \begin{equation}\label{cond:radius:time}
    \left\{
    \begin{split}
      & R=2\|(v_0,v_1)\|_{\cH^1}\\
      & t_{\ast}^{\frac{5-p}{2}}R^{p-1}\ll1\\
      & t_{\ast}^{\frac{5-p}{2}}K^{p}t_{\ast}^{p\beta}\ll R\\
      & t_{\ast}^{\frac{5-p}{2}}K^{p-1}t_{\ast}^{(p-1)\beta}\ll 1.
    \end{split}
    \right.
  \end{equation}
  All of these conditions can be guaranteed by selecting $t_{\ast}=c(\|(v_0,v_1)\|_{\cH^1}+K)^{-\gamma}$ with $\gamma$ positive for some sufficiently small $c>0$. This finishes the proof by the Banach contraction mapping principle.
\end{proof}

Now we are going to construct solutions to Equation \eqref{eq:wave}. By denoting $v:=u-f$ with $f=S(t)(u_0,u_1)$, then $v$ satisfies the following zero-initial data problem
 \begin{equation}
        \label{eq:wave:0ini}
        \left\{
          \begin{split}
          &(\ptd^2_t-\Delta)v+|v+f|^{p-1}(v+f)=0\\
          &\big(v,\ptd_tv\big)|_{t=0}=(0,0)
        \end{split}
        \right.
 \end{equation}

 The following deterministic result, allows us to draw an \textit{a priori} energy bound for solution $v$ to \eqref{eq:wave:0ini} with $f=S(t)(u^{\omega}_0,u^{\omega}_1)$ from that to solution $v_N$ to the truncated equation \eqref{eq:truncated:v}.

 \begin{prop}\label{prop:energy:v}
Let $f_N:=P_{\leq N}f$ denote the projection onto the first $N$-Fourier modes of the given function $f$ and $v_N$ be the solution to the truncated wave equation (\ref{eq:truncated:v}). Given finite $T>0$, assume the following conditions hold:
  \begin{enumerate}[(i)]
  \item There exists $K>0$ for some $\beta>0$ such that
    \begin{equation}
      \label{eq:strich:small}
      \|f\|_{L^{\frac{2p}{p-3}}_tL^{2p}_x(I\times \bT^3)}\leq K|I|^{\beta}
    \end{equation}
    for any compact interval $I\subset [0,T]$.
  \item For each dyadic $N\geq 1$, a solution $v_N$ to \eqref{eq:truncated:v} exists on $[0,T]$ and satisfies uniform \textit{a priori} energy bound
    \begin{equation}
      \label{eq:energy:bound}
      \sup_N\sup_{t\in[0,T]}\|(v_N(t),\ptd_tv_N(t)\|_{\cH^1(\bT^3)}<C_0(T)<\infty.
    \end{equation}
  \item There holds for any dyadic $N\geq1$ and some $\alpha>0$
    \begin{equation}
      \label{eq:tail:decay}
      \|f-f_N\|_{L^{\frac{2p}{p-3}}_TL^{2p}_x}\leq C_1(T)N^{-\alpha}.
    \end{equation}
  \end{enumerate}
  Then there exists a unique solution $(v,\ptd_tv)\in C([0,T];\cH^1(\bT^3))$ to \eqref{eq:wave:0ini} satisfying
  \begin{equation}
    \label{eq:energy:bound:glo}
    \sup_{t\in[0,T]}\|(v(t),\ptd_tv(t))\|_{\cH^1(\bT^3)}<2C_0(T)<\infty.
  \end{equation}
\end{prop}

\begin{proof}
  To prove Proposition \ref{prop:energy:v}, we need the following lemma, which states that we can solve simultaneously, on some time interval $[t_0,t_{\ast}]$ for any $t_0\in[0,T)$, the following two equations
  \begin{equation}
    \label{eq:truncated:vN:non0}
    \left\{
      \begin{split}
        &(\ptd^2_t-\Delta)v_N+|v_N+f_N|^{p-1}(v_N+f_N)=0\\
        &\big(v_N,\ptd_tv_N\big)|_{t=t_0}=\big(v_N(t_0),\ptd_tv_N(t_0)\big)
      \end{split}
      \right.
    \end{equation}
    and
    \begin{equation}
      \label{eq:truncated:v:non0}
      \left\{
        \begin{split}
          &(\ptd_t^2-\Delta)v+|f+v|^{p-1}(f+v)=0\\
          &\big(v,\ptd_tv\big)|_{t=t_0}=\big(v(t_0),\ptd_tv(t_0)\big).
        \end{split}
        \right.
    \end{equation}
  \begin{lem}
    \label{lem:sol:simut}
    Assume there hold \eqref{eq:strich:small}, \eqref{eq:energy:bound}, \eqref{eq:tail:decay}. Assume also there holds for any $t_0\in[0,T)$
    \begin{equation}
      \label{eq:bound:energy:v:loc}
      \sup_{t\in[0,t_0]}\|(v,\ptd_tv)\|_{\cH^1}<2C_0(T)<\infty,
    \end{equation}
    where $C_0(T)$ is the same constant showing its face in \eqref{eq:energy:bound}. Then there exist a sufficiently large $N_1$ and a positive time $t_{\ast}=t_{\ast}(C_0,K,N_1)>0$ such that, for all $N\geq N_1$, on the time interval $I=[t_0,t_0+t_{\ast}]$, we can solve simultaneously the equations \eqref{eq:truncated:vN:non0} and \eqref{eq:truncated:v:non0} and denote these solutions as $v_N$, $v$ respectively. Moreover, we have
    \begin{equation}
      \label{cond:small}
   t_{\ast}^{\frac{5-p}{2}}(\|v\|^{p-1}_{L^{\frac{2p}{p-3}}_IL^{2p}_x}+\|v_N\|^{p-1}_{L^{\frac{2p}{p-3}}_IL^{2p}_x}+\|f\|^{p-1}_{L^{\frac{2p}{p-3}}_IL^{2p}_x}+\|f_N\|^{p-1}_{L^{\frac{2p}{p-3}}_IL^{2p}_x})\ll1,
  \end{equation}
  for all $N\geq N_1$.
\end{lem}

\begin{proof}[Proof of Lemma \ref{lem:sol:simut}]
  We also use the fixed point argument as we did in the proof of Lemma \ref{lem:key:loc}. Thus, we only outline the mains steps here.  Define the maps $L_1$ on $B(0,R_1)\subset X$ and $L_2$ on $B(0,R_2)$ respectively as:
  \begin{eqnarray*}
    &L_1:u_N\in B(0,R_1)\longmapsto v_N\\
    &L_2:u\in B(0,R_2)\longmapsto v,
  \end{eqnarray*}
  where $v_N$ and $v$ solves respectively the equations
  \begin{equation*}
    \left\{
      \begin{split}
        &(\ptd^2_t-\Delta)v_N+|u_N+f_N|^{p-1}(u_N+f_N)=0\\
        &\big(v_N,\ptd_tv_N\big)|_{t=t_0}=\big(v_N(t_0),\ptd_tv_N(t_0)\big)
      \end{split}
      \right.
    \end{equation*}
    and
    \begin{equation*}
      \left\{
        \begin{split}
          &(\ptd_t^2-\Delta)v+|f+u|^{p-1}(f+u)=0\\
          &\big(v,\ptd_tv\big)|_{t=t_0}=\big(v(t_0),\ptd_tv(t_0)\big).
        \end{split}
        \right.
      \end{equation*}

      By \eqref{eq:strich:small} and \eqref{eq:tail:decay}, we have
      \begin{equation*}
        \|f_N\|_{L^{\frac{2p}{p-3}}_IL^{2p}_x}\leq K|I|^{\beta}+C_1(T)N^{-\alpha}.
      \end{equation*}
      In order for $L_1$ and $L_2$ to be contracting maps onto $B(0,R_1)$ and $B(0,R_2)$ respectively, we do the same calculations as we did in Lemma \ref{lem:sol:simut}. And finally we can assume
   \begin{equation}\label{cond:radius:time:R1}
    \left\{
    \begin{split}
      & R_1=2C_0(T)\\
      & t_{\ast}^{\frac{5-p}{2}}R_1^{p-1}\ll1\\
      & t_{\ast}^{\frac{5-p}{2}}(K^{p}t_{\ast}^{p\beta}+N^{-p\alpha})\ll R_1\\
      & t_{\ast}^{\frac{5-p}{2}}(K^{p-1}t_{\ast}^{(p-1)\beta}+N^{-(p-1)\alpha})\ll 1.
    \end{split}
    \right.
  \end{equation}
  and
   \begin{equation}\label{cond:radius:time:R2}
    \left\{
    \begin{split}
      & R_1=2C_0(T)\\
      & t_{\ast}^{\frac{5-p}{2}}R_2^{p-1}\ll1\\
      & t_{\ast}^{\frac{5-p}{2}}K^{p}t_{\ast}^{p}\ll R_2\\
      & t_{\ast}^{\frac{5-p}{2}}K^{p-1}t_{\ast}^{p-1}\ll 1.
    \end{split}
    \right.
  \end{equation}
  Thus there exists sufficiently large $N_1=N_1(K,C_0(T))$ such that, for all $N\geq N_1$, by choosing $t_{\ast}=c(K+C_0(T))^{-\gamma}$ with $c$ and $\gamma$ small positive constants, we guarantee these two assumptions hold true at the same time. By choosing $t_{\ast}$ even smaller if necessary, we can validate the estimate \eqref{cond:small}. 
\end{proof}
As a consequence of Lemma \ref{lem:sol:simut}, we have for the difference $w_N=v-v_N$ on the time interval $I=[t_0,t_0+t_{\ast}]$:
\begin{equation}
    \label{eq:remainder:I}
    \|w_N\|_{L^{\infty}_I\cH^1}+\|w_N\|_{L^{\frac{2p}{p-3}}_IL^{2p}_x}\leq C_2\|w_N(t_0)\|_{\cH^1}+\frac{1}{2}\|w_N\|_{L^{\frac{2p}{p-3}}_IL^{2p}_x}+\frac{1}{2}\|f-f_N\|_{L^{\frac{2p}{p-3}}_IL^{2p}_x}.
  \end{equation}
  Thus we have
  \begin{equation}
    \label{eq:energy:remainder:I}
    \|w_N\|_{L^{\infty}_I\cH^1}+\|w_N\|_{L^{\frac{2p}{p-3}}_IL^{2p}_x}\leq C_3(T)(\|w_N(t_0)\|_{\cH^1}+N^{-\alpha})
  \end{equation}
  for all $N\geq N_1$.
  
Now we begin to solve Equation \eqref{eq:wave:0ini} with $t_0=0$. As $\|(v,\ptd_tv)\|_{\cH^1}(0)=0<2C_0(T)$, we can solve simultaneously the equations \eqref{eq:truncated:vN:non0} and \eqref{eq:truncated:v:non0} on the time interval $I_0=[0,t_{\ast}]$, where $t_{\ast}$ is obtained in Lemma \ref{lem:sol:simut} and it depends only on $C_0(T)$ and $K$. Furthermore, by \eqref{eq:remainder:I} and \eqref{eq:energy:remainder:I}, we have for all $N\geq N_1$
 \begin{equation}
      \label{cond:small:1}
      t_{\ast}^{\frac{5-p}{2}}(\|v\|^{p-1}_{L^{\frac{2p}{p-3}}_{I_0}L^{2p}_x}+\|v_N\|^{p-1}_{L^{\frac{2p}{p-3}}_{I_0}L^{2p}_x}+\|f\|^{p-1}_{L^{\frac{2p}{p-3}}_{I_0}L^{2p}_x}+\|f_N\|^{p-1}_{L^{\frac{2p}{p-3}}_{I_0}L^{2p}_x})\ll1.
    \end{equation}
    and hence
    \begin{equation}
    \label{eq:remainder:I0}
    \|w_N\|_{L^{\infty}_{I_0}\cH^1}+\|w_N\|_{L^{\frac{2p}{p-3}}_{I_0}L^{2p}_x}\leq C_2\|w_N(0)\|_{\cH^1}+\frac{1}{2}\|w_N\|_{L^{\frac{2p}{p-3}}_{I_0}L^{2p}_x}+\frac{1}{2}\|f-f_N\|_{L^{\frac{2p}{p-3}}_{I_0}L^{2p}_x}.
  \end{equation}
  Thus we have
  \begin{equation}
    \label{eq:remainder:I0}
    \|w_N\|_{L^{\infty}_{I_0}\cH^1}+\|w_N\|_{L^{\frac{2p}{p-3}}_{I_0}L^{2p}_x}\leq C_3(T)N^{-\alpha}
  \end{equation}
  Therefore, by \eqref{eq:remainder:I0} and \eqref{eq:energy:bound}, there exists $N_2=N_2(T)\geq N_1$ such that
  \begin{equation}
    \label{eq:energy:v:I0}
    \|(v,\ptd_tv)\|_{L^{\infty}_{I_0}\cH^1}\leq C_0(T)+TC_3(T)N^{-\alpha}<2C_0(T)
  \end{equation}
  for all $N\geq N_2$.

  This last bound \eqref{eq:energy:v:I0} allows us to apply Lemma \ref{lem:sol:simut} again with $t_0=t_{\ast}$. And by denoting $I_1=[t_{\ast},2t_{\ast}]$, we have
  \begin{equation*}
      t_{\ast}^{\frac{5-p}{2}}(\|v\|^{p-1}_{L^{\frac{2p}{p-3}}_{I_1}L^{2p}_x}+\|v_N\|^{p-1}_{L^{\frac{2p}{p-3}}_{I_1}L^{2p}_x}+\|f\|^{p-1}_{L^{\frac{2p}{p-3}}_{I_1}L^{2p}_x}+\|f_N\|^{p-1}_{L^{\frac{2p}{p-3}}_{I_1}L^{2p}_x})\ll1.
    \end{equation*}
    Similar argument as we did on $I_0$, there exists $N_3=N_3(T)\geq N_2$ such that
    \begin{equation}
      \label{eq:energy:v:I1}
      \|(v,\ptd_tv)\|_{L^{\infty}_{I_1}\cH^1}\leq C_0(T) + TC_3(T)(C_3(T)+1)N^{-\alpha}<2C_0(T)
    \end{equation}
    for all $N\geq N_3$. Notice that the bound \eqref{eq:energy:v:I1} together with \eqref{eq:energy:v:I0} allows us to use Lemma \ref{lem:sol:simut} again.

    Iterate the above procedure, we can extend the solution $v$ onto the whole interval $[0,T]$. Moreover, there exists $N_0=N_0(T,t_{\ast})\in\mathbb{N}$ such that
    \begin{equation*}
      \sup_{t\in[0,T]}\|(v,\ptd_tv)\|_{\cH^1}\leq C_0(T)+T(C_3(T)+1)^{\big[\frac{T}{t_{\ast}}\big]}N^{-\alpha}<2C_0(T)
    \end{equation*}
    for all $N\geq N_0$. Hence we have that the solution $v$ to Equation \eqref{eq:wave:0ini} satisfies the energy bound \eqref{eq:energy:bound:glo} on $[0,T]$.
\end{proof}

%%%%%%%%%%%%%%%%%%%%%%%%%%%%%%%%%%%%%%%%%%%%%%%%%%%%%%%%%%%%%%%
%% almost sure global existence results
%%%%%%%%%%%%%%%%%%%%%%%%%%%%%%%%%%%%%%%%%%%%%%%%%%%%%%%%%%%%%%%

\section{Almost surely global well-posedness}
The following proposition can finish the proof of Theorem \ref{sec:introduction:thm:main}, see \cite{Burq2011} and \cite{Poc} for details.
\begin{prop}
  [Almost sure global well-posedness]
  Given $s\in(\frac{p-3}{p-1},1)$, for any data $(u_0,u_1)\in\cH^s$, let $(u_0^{\omega},u^{\omega}_1)$ be the randomization defined in (\ref{data.rando}) under the assumption (\ref{main.assum}). Then given any $T,\eps>0$, there exists $\Omega_{T,\eps}\subset \Omega$ such that
  \begin{enumerate}[(i)]
  \item $\bP(\Omega^{c}_{T,\eps})<\eps$,
  \item For any $\omega\in\Omega_{T,\eps}$, there exists a unique solution $u^{\omega}$ to Equation (\ref{eq:wave}) with $(u^{\omega},\ptd_tu^{\omega})|_{t=0}=(u^{\omega}_0,u^{\omega}_1)$ in the class:
    \begin{equation*}
      \big(S(t)(u^{\omega}_0,u^{\omega}_1),\ptd_tS(t)(u^{\omega}_0,u^{\omega}_1)\big)+C([0,T];\cH^1(\bT^3))\subset C([0,T];\cH^s(\bT^3)).
    \end{equation*}
  \item For any $\omega\in\Omega_{T,\eps}$, the following probabilistic energy bound holds for the nonlinear part $v^{\omega}$ of the solution $u^{\omega}$:
    \begin{equation*}
      \sup_{t\in[0,T]}\|(v^{\omega},\ptd_tv^{\omega})\|_{\cH^1(\bT^3)}<C(T,\eps,\|(u_0,u_1)\|_{\cH^s(\bT^3)}).
    \end{equation*}
  \end{enumerate}
\end{prop}
\begin{proof} We also argue in the same way as in \cite{OhPo}. We first construct a set $\Omega_1$, over which the assumption $(iii)$ in Proposition \ref{prop:energy:v} holds for all dyadic $N$.
  As usual, $z^{\omega}=S(t)(u_0^{\omega},u_1^{\omega})$ and $z_N^{\omega}=P_{\leq N}S(t)(u^{\omega}_0,u^{\omega}_1)$. Taking $\alpha\in(0,s)$, set
  \begin{equation*}
    M=M(T,\eps,\|(u_0,u_1)\|_{\cH^{\alpha}})\sim T^{\frac{p-3}{p}}\Big(\log\frac{1}{\eps}\Big)^{\frac{1}{2}}\|(u_0,u_1)\|_{\cH^{\alpha}}.
  \end{equation*}
  Then denote
  \begin{equation*}
    \Omega_1:=\Omega_1(T,\eps):=\{\omega\in\Omega: \|\lan\nabla\ran^{\alpha}z^{\omega}\|_{L^{\frac{2p}{p-3}}_TL^{2p}_x}\leq M\}.
  \end{equation*}
  By Lemma \ref{prop:local-in-time:strich} (ii) that
  \begin{equation}
    \label{eq:proba:omega1}
    \bP(\Omega^c_1)<\frac{\eps}{3}.
  \end{equation}
  Moreover, for each $\omega\in\Omega_1$, we have for any $N\geq 1$
  \begin{equation}
    \label{eq:bound:remainder:N}
    \|z^{\omega}-z^{\omega}_N\|_{L^{\frac{2p}{p-3}}_T{L^{2p}_x}}\leq N^{-\alpha/2}\|\lan\nabla\ran^{\alpha}z^{\omega}\|_{L^{\frac{2p}{p-3}}_TL^{2p}_x}\leq MN^{-\alpha/2}.
  \end{equation}

  Next, we are going to construct another subset $\Omega_2\subset\Omega$, over which the assumption $(ii)$ in Proposition \ref{prop:energy:v} holds for all dyadic $N$. Given any dyadic $N\geq 1$, apply Proposition \ref{lemma:key:energy}, we can construct $\Omega_2(N):=\tilde{\Omega}_{N,T,\eps}$ with
  \begin{equation}
    \label{eq:proba:Omega2}
    \bP(\Omega_2^c)<\frac{\eps}{3}
  \end{equation}
  such that
  \begin{equation}
    \label{eq:energy:bound:Omega2}
    \sup_{t\in[0,T]}\|(v^{\omega}_N(t),\ptd_tv^{\omega}_N(t))\|_{\cH^1}<C_0(T,\eps,\|(u_0,u_1)\|_{\cH^s})<\infty
  \end{equation}
  for each $\omega\in\Omega_2(N)$. The main point here is that the constant $C_0=C_0(T,\eps,\|(u_0,u_1)\|_{\cH^s})$ can be chosen independent of $N$.

  Lastly, fix $K=\|(u_0,u_1)\|_{\cH^0}$ and $2\beta=\frac{p-3}{2p}$ in the following. Let $t_{\ast}>0$ be a small number and be chosen later. By writing $[0,T]=\cup^{[T/t_{\ast}]}_{k=0}I_k$ with $I_k=[kt_{\ast},(k+1)t_{\ast}]\cap[0,T]$, define $\Omega_3$ by
  \begin{equation}
    \label{eq:Omega3}
    \Omega_3:=\bigcup_{k=0}^{[\frac{T}{t_{\ast}}]}\Big\{\omega\in \Omega:\|z^{\omega}\|_{L^{\frac{2p}{p-3}}_{I_k}L^{2p}_x}\leq K|I_k|^{\beta}\Big\}.
  \end{equation}
  Then by Lemma \ref{prop:local-in-time:strich} with $|I_k|\leq t_{\ast}$, we have
  \begin{equation*}
    \bP(\Omega^c_3)\leq \sum_{k=0}^{[T/t_{\ast}]}\bP\Big(\|z^{\omega}\|_{L^{\frac{2p}{p-3}}_{I_k}L^{2p}_x}>K|I_k|^{\beta}\Big)\leq \exp\Big(-\frac{c}{T^2t_{\ast}^{\beta}}\Big).
  \end{equation*}
  By taking $t_{\ast}$ even smaller if necessary, we have
  \begin{equation*}
    \bP(\Omega^c_3)\leq \frac{T}{t_{\ast}}t_{\ast}\exp\Big(-\frac{c}{2T^2t_{\ast}^{\frac{p-3}{2p}}}\Big)=T\exp\Big(-\frac{c}{2T^2t_{\ast}^{\beta}}\Big).
  \end{equation*}
  Hence, by choosing $t_{\ast}=t_{\ast}(T,\eps)$ sufficiently small, we have
  \begin{equation}
    \label{eq:proba:Omega3}
    \bP(\Omega^c_3)<\frac{\eps}{3}.
  \end{equation}

  Let $\Omega_{T,\eps}:=\Omega_1\cap\Omega_2(N_0)\cap\Omega_3$, where $N_0$ is to be chosen later. Then from (\ref{eq:proba:omega1}),~(\ref{eq:proba:Omega2}) and (\ref{eq:proba:Omega3}), we have that
  \begin{equation*}
    \bP(\Omega^c_{T,\eps})<\eps.
  \end{equation*}
  By choosing $N_0=N_0(T,\eps,\|(u_0,u_1)\|_{\cH^s})\gg1$, by Proposition \ref{prop:energy:v}, we have that there exists a solution $v^{\omega}$ to Equation (\ref{eq:wave:0ini}) on $[0,T]$ for each $\omega\in\Omega_{T,\eps}$. Hence for $\omega\in\Omega_{T,\eps}$, there exists a solution $u^{\omega}=z^{\omega}+v^{\omega}$ to Equation (\ref{eq:wave}) on $[0,T]$. Moreover, there holds the estimate:
  \begin{equation*}
    \sup_{t\in[0,T]}\|(v^{\omega}(t),\ptd_tv^{\omega}(t))\|_{\cH^1(\bT^3)}<2C_0(T,\eps,\|(u_0,u_1)\|_{\cH^s(\bT^3)})<\infty,
  \end{equation*}
  for each $\omega\in\Omega$
\end{proof}

\bibliographystyle{plain}
\bibliography{pwave}

\end{document}